\documentclass[12pt,reqno]{amsart}

\usepackage{latexsym}
\usepackage{graphicx}

\usepackage{amssymb}

\renewcommand{\epsilon}{\varepsilon}

\newcommand{\ma}{{\operatorname{MA}}}

\newcommand{\poly}{{\operatorname{Poly}}}

\newcommand{\T}{{\mathbf T}^m}

\newcommand{\sm}{\setminus}
\newcommand{\szego}{Szeg\"o }

\newcommand{\Si}{\Sigma}
\newcommand{\inv}{^{-1}}
\newcommand{\kahler}{K\"ahler }

\newcommand{\wt}{\widetilde}
\newcommand{\wh}{\widehat}
\newcommand{\PP}{{\mathbb P}}
\newcommand{\N}{{\mathbb N}}
\newcommand{\R}{{\mathbb R}}
\newcommand{\C}{{\mathbb C}}

\newcommand{\Z}{{\mathbb Z}}

\newcommand{\CP}{\C\PP}
\renewcommand{\d}{\partial}
\newcommand{\dbar}{\bar\partial}
\newcommand{\ddbar}{\partial\dbar}

\newcommand{\E}{{\mathbf E}}

\renewcommand{\phi}{\varphi}

\newcommand{\ccal}{\mathcal{C}}
\newcommand{\dcal}{\mathcal{D}}

\newcommand{\lcal}{\mathcal{L}}

\newcommand{\ncal}{\mathcal{N}}
\newcommand{\ocal}{\mathcal{O}}
\newcommand{\pcal}{\mathcal{P}}

\newcommand{\scal}{\mathcal{S}}

\newcommand{\al}{\alpha}
\newcommand{\be}{\beta}
\newcommand{\ga}{\gamma}

\newcommand{\la}{\lambda}
\newcommand{\ep}{\varepsilon}

\newcommand{\De}{\Delta}
\newcommand{\om}{\omega}

\newcommand{\half}{{\frac{1}{2}}}
\newcommand{\vol}{{\operatorname{Vol}}}

\newcommand{\SU}{{\operatorname{SU}}}
\newcommand{\FS}{{{\operatorname{FS}}}}
\newcommand{\supp}{{\operatorname{Supp\,}}}
\renewcommand{\phi}{\varphi}

\newtheorem{theo}{{\sc Theorem}}[section]
\newtheorem{maintheo}{{\sc Theorem}}
\newtheorem{cor}[theo]{{\sc Corollary}}
\newtheorem{maincor}[maintheo]{{\sc Corollary}}
\newtheorem{lem}[theo]{{\sc Lemma}}
\newtheorem{prop}[theo]{{\sc Proposition}}

\newenvironment{rem}{\medskip\noindent{\it Remark:\/} }{\medskip}
\newenvironment{defin}{\medskip\noindent{\it Definition:\/} }{\medskip}

\title[Random complex fewnomials, I] {Random
complex fewnomials, I }

\author{Bernard Shiffman}
\address{Department of Mathematics, Johns Hopkins University, Baltimore, MD
21218, USA} \email{shiffman@math.jhu.edu}

\author{Steve Zelditch}
\address{Department of Mathematics, Northwestern University, Evanston, IL 60208, USA} \email{zelditch@math.northwestern.edu}

\thanks{Research of the first author partially supported by NSF grant DMS-0901333; research of the
second author partially supported by NSF grant  DMS-0904252.}

\date{\today}

\begin{document}

\begin{abstract} We introduce several notions of `random
fewnomials', i.e. random polynomials with a fixed number $f$ of
monomials of degree $N$. The $f$ exponents are chosen at random and then the
coefficients are chosen to be Gaussian random, mainly from the
$\SU(m + 1)$ ensemble.  The results give limiting formulas as $N
\to \infty$ for the expected distribution of complex zeros of a
system of $k$ random fewnomials in $m$ variables ($k \leq m$).
When $k = m$, for $\SU(m + 1)$ polynomials,  the limit is the
Monge-Amp\`ere measure of a toric \kahler potential on $\CP^m$
obtained by averaging a `discrete Legendre transform'  of the Fubini-Study symplectic potential at  $f$ points of the unit simplex $\Sigma\subset\R^m$.

\end{abstract}

\maketitle

\tableofcontents

\section*{Introduction}

This article is concerned with the distribution of complex zeros
of random systems of {\it fewnomials}. Fewnomials are polynomials
$$\sum_{\alpha: |\alpha| \leq N} c_{\alpha} z^{\alpha}: \;\;
\#\{\alpha: c_{\alpha} \not= 0 \} = f << N, $$
 with a relatively small number  $f << N$ of non-zero coefficients in comparison
 to the degree. For instance, $z^{10, 000} + 5 z^{1,000} - 3 z^{100} - 1$ is
 a fewnomial but $z^{10, 000} + z^{9, 999} + \cdots + z + 1$ is not. The fundamental
 idea is that the number of monomials, rather than the degree, measures the complexity
 of a polynomial system \cite{KH}.
 The purpose of this article to begin an investigation of fewnomial complexity
 bounds from a probabilistic viewpoint. In this article, we  introduce several natural
 ensembles of  `random fewnomial systems',  and study the expected distribution  of their complex zeros.
 In subsequent articles, we plan to study real zeros of real fewnomials and the more difficult
 problem of the correlations and variance of both real and complex
 zeros.
 The overall purpose is to study  from
  a statistical point of view
  Khovanskii's bounds \cite{KH} on the Betti numbers of real algebraic varieties given by fewnomials
  and on the number of real zeros of a full fewnomial system. Statistical properties of the topology of real algebraic varieties given by non-fewnomial polynomials (i.e., with all the coefficients nonzero) have been studied, for example, in \cite{Bu,Ro,ShSm,GW} and the references in these papers, but not much is known statistically about zeros of fewnomials.

To put our problem into context, let  us recall Khovanskii's
theorem:
 Let $P=
(P_1, \dots, P_m)$ denote a system of $m$ complex polynomials on
$(\C^*)^m$, and let $\Delta_j=\De_{P_j}$ denote the Newton polytope of
$P_j$, i.e the convex hull of the exponents appearing
non-trivially in $P_j$. Let $U \subset {\mathbf T}^m$ be an open set,
where ${\mathbf T}^m \subset (\C^*)^m$ is the real $m$-torus, and let
$N(P, U)$ be the number of zeros with arguments lying in $U$. When
$U = {\mathbf T}^m$, $N(P, U)$ counts the total number of zeros in $(\C^*)^m$ , which
by the Bernstein-Kouchnirenko theorem \cite{Be,Ko} can be expressed in terms of
the mixed volume $V(\Delta_1, \dots, \Delta_m)$.  Given an
angular sector determined by $U$, the number \begin{equation}
\label{BKK} S(P, U) := V(\Delta_1, \dots, \Delta_m)
\vol(U)/\vol({\mathbf T}^m) \end{equation} may be viewed as the `average
number' of complex zeros in the sector among random  polynomial systems
with the prescribed Newton polytopes $\Delta_j$. We denote this
class of polynomial systems by
$$P \in \poly{(\Delta_1, \dots, \Delta_m)}: = \{(P_1, \dots,
P_m): \;\; \Delta_{P_j} = \Delta_j\}. $$  Khovanskii's complex
fewnomials theorem \cite[\S 3.13,~Th.~2]{KH}  asserts that
\begin{equation} \label{KHOV} \sup_{P \in \poly{(\Delta_1, \dots, \Delta_m)}} \; |N(P, U) - S(P, U)| \leq \Pi (U, \Delta_1, \dots,
\Delta_m) \phi(m, f)
\end{equation} where $f$ is the number of non-zero coefficients of the system, and where
$\Pi (U, \Delta_1, \dots, \Delta_m) $ is the smallest number of
translates of a certain region $\Delta^*\subset{\mathbf T}^m$ required to cover the boundary of $U$.
One of the principal applications of this result is to give an
upper bound for the number $N_{\R}(P)$  of  real zeros of a
fewnomial system: If $U_j$ is a sequence of small balls around
$\{0\}$ shrinking to the point $\{0\}$, one has $\Pi = 1$ and
$S(P, U_j) \to 0$ and one obtains a bound of the form
\begin{equation} \label{BOUND} |N_{\R} (P) | \leq  \phi(m, f)
\end{equation}
entirely  in terms of the number of non-zero monomials appearing
in it and not its degree.
 We will refer to $f$
as the {\it fewnomial number} of the system.

Khovanskii's result may be interpreted in terms of the angular
projection $\mbox{Arg}: (z_1, \dots, z_m) = (\frac{z_1}{|z_1|},
\dots, \frac{z_m}{|z_m|})$ of the zero set to the real torus
${\mathbf T}^m$. His result (in the full system case) says that the
angular projection of the fewnomial zero set is rather evenly
distributed in ${\mathbf T}^m$. As a result, not too many zeros
concentrate on the real set where $\theta = 0$. Note that his
measure of the concentration, taking the supremum in (\ref{KHOV}),
is very astringent and is governed by the extreme cases. The idea
of our work is to study its  average value over fewnomial systems
and polynomials.

The motivation for the statistical study  is that the known
estimates of $\phi(m, f)$ are very large and are widely
conjectured to overestimate the the bound by many orders of
magnitude. Khovanskii's bound states that $ \phi(m, f) \leq 2^m
2^{f (f - 1)/2} (m + 1)^f$. See \cite{BBS,BRS} for relatively
recent bounds and \cite{So,S,St2}  for further background. A
conjecture of Kouchnirenko, as corrected and refined by a number
of people, states that the maximum number of real zeros in the positive
real quadrant should be roughly $|f|^{2 m}$ where $|f|$ is the
total  number of monomials in the system. The uncertainty as to
the true order of magnitude of $\phi(m, f)$ suggests studying the
bound probabilistically. The bound (\ref{BOUND}) resembles a
variance estimate although it is measured in  the much more
difficult sup norm. It reflects the extremal behavior, which may
only occur very rarely. This raises the question, what is the
expected or average  order of magnitude of the variance?

In this article we begin the study of random fewnomial systems by
introducing several  probability measures on spaces of complex fewnomials---
i.e., on the set of pairs $(S, P)$ of spectra and polynomial
systems with the given spectra. Our main results give the expected
limit distribution of complex zeros in the ensembles. For example, Theorem \ref{main} says that for a random system
$(P^N_1, \dots, P^N_m)$  of
fewnomials on $\C^m$, each of fewnomial number $f$ and of degree $N$, where the exponents are chosen uniformly at random and  the coefficients are chosen at random from  the
$\SU(m + 1)$ ensemble (described below), the expected distribution of zeros in
$(\C^*)^m$ is asymptotic to
$$ N^m  \det_{1\le p,q \le m} \left(\frac{\d^2}{\d\rho_p\d\rho_q}\int_{\Sigma^f}
 \max_{j = 1, \dots, f} \left[\langle \rho, \lambda^j
\rangle -\langle
\wh{\lambda^j\,}, \log \wh{\,\lambda^j\, } \rangle  \right]\, d \lambda^1 \cdots d \lambda^f\right)d\rho_1\cdots d\rho_m \frac{d\theta_1}{2\pi}\cdots \frac{d\theta_m}{2\pi} , $$
where $z=(e^{\rho_1/2+i\theta_1},\dots,e^{\rho_m/2+i\theta_m} )$, $\Sigma$ is the unit simplex in $\R^m$ with probability measure $d\la={m!} \,d\la_1\cdots d\la_m$, and $\wh \la= (1-|\la|,\la_1,\dots,\la_m)$.

\subsection{Fewnomial ensembles}

  We consider
several natural definitions which are motivated by different kinds
of applications. More precise and detailed definitions are given in
\S \ref{FEWENSEMBLES}.

We denote the space of all complex holomorphic polynomials of
degree $N$ by $\poly(N)$.
  By the {\it spectrum} (or {\it support}) of
a polynomial $P$, we mean the set $S_P$ of exponents of
its non-zero monomials. We denote the space of polynomials with
spectrum {\it contained in} $S$ by
\begin{equation} \poly(S) = \{P(z_1, \dots, z_m) = \sum_{\alpha \in
S} c_{\alpha} \chi_{\alpha}(z), \;\;\chi_{\alpha}(z) :=
z^{\alpha}\}, \;\; S\subset \N^m.
\end{equation}
The {\it Newton polytope} of $P$ is the  convex hull $\Delta_P$ of the spectrum $S_P$.  More generally, we consider
 a system of $k \leq m$ polynomials $P_1, \dots, P_k$
in $m$ complex variables, and write
\begin{equation} \poly(S_1, \dots, S_k) =\{(P_1, \dots, P_k): \;\;
P_j \in \poly(S_j)\}.
\end{equation}
When $k = m$ we speak of a `full' system, where the simultaneous zeros are almost always zero-dimensional.

In all of our definitions of random fewnomial system,  the numbers
$f_j$ of elements of each spectrum $S_j$ and the degrees $N$ of
the polynomials $P_j$  are fixed. We then randomize with respect
to the spectra $S_j$ and with respect to the coefficients $c_{j
\alpha}.$ With regard to the spectra, there are several natural
choices of probability measure:

\begin{itemize}

\item [(I)] Fixed spectrum up to dilation: Here, we fix a spectrum
$S$, and then dilate it deterministically as the degree $N$ grows,
i.e. scale $S \to NS$. This notion of random fewnomial is
analogous to our notion of random polynomial with fixed Newton
polytope in \cite{SZ2}. The main difference is that the `polytope'
which we dilate is non-convex; indeed, it just  consists of a
fixed set of $f$ points. In fact, the  techniques of \cite{SZ2}
generalize quite naturally to all  non-convex polytopes. The only
randomness is then with respect to the coefficients. The result is
given in Theorem \ref{EZa}.

\item [(II)] Dilates of a random spectrum from a polytope
$\Delta$: in this ensemble, we fix $\Delta$, choose the spectrum
at random from $\Delta$ and then dilate the resulting spectrum. It
is only a small step from case (I), but is apparently important in
computational work. (We thank Maurice Rojas for emphasizing the
interest of this case.) The result is
given in Corollary \ref{EZc}.

\item [(III)] Random spectra of degree $N$:  At the opposite extreme, we
may choose the spectra completely randomly (with respect to
counting measure) from all possible $f$-element sets of exponents $\al\in\N^m$ of length $|\al|=|\al_1|+\dots+|\al_m|\le N$---i.e., subsets of the integral
simplex $\Z^m\cap N \Sigma$, where $\Sigma=\{x\in\R^m:x_j\ge 0, \sum x_j\le 1\}$ is the unit $m$-simplex,
 and  $N \Sigma$ is its dilate by $N$. We put uniform
measure on $\Z^m \cap N \Sigma$ and then choose spectra $S_j\subset \Z^m\cap N \Sigma$ of
fixed cardinality $f$ independently and uniformly relative to
counting measure. The result is
given in Theorem \ref{main} for $\SU(m+1)$ fewnomials and in Theorem \ref{maintoric} for general toric \kahler potentials.

\item [(IV)] Random spectra contained in  fixed Newton polytopes: Rather than just consider the
simplex,  and motivated by Khovanskii's variation formula, we now
fix $k$ convex lattice polytopes $\Delta_1,\dots,\Delta_k$ and then choose random
spectra $S_j \subset \Delta_j \cap \Z^m$ with fixed cardinalities
$f_j$ independently with uniform measures  from these polytopes. We
then replace the $\Delta_j$ by their dilations by $N\Delta_j$ to obtain higher degree ensembles. We note that the convex hull of $S_j$ is contained in $ \Delta_j$ but
equality rarely occurs. The result is
given in Theorem \ref{mainIV}.

\end{itemize}

Having decided on an ensemble of spectra $\alpha$, we then define
probability measures on the coefficients $c_{\alpha}$. We only
consider Gaussian probability measures and make standard choices
which are consistent with Khovanskii's bound. A key point is that
Gaussian measures are determined by inner products in the space of
polynomials. We choose the inner products as in \cite{SZ,SZ3,SoZ,SoZ3}
to be those $G_N(\phi, \nu)$  of the weighted $L^2$ spaces
$L^2(\CP^m, e^{- N \phi} d\nu)$ of pluri-potential theory, which
are specified by a \kahler potential $\phi$ or Hermitian metric $h
= e^{- \phi}$ and a measure $d\nu$ on $\CP^m$. It is natural to
restrict to $\phi, \nu$ which are {\it toric}, i.e. invariant
under the standard ${\mathbf T}^m$ torus action on $\CP^m$. Then the
monomials $\{z^{\alpha}\}$ are always orthogonal and the Gaussian
ensembles only differ in the $L^2$- norms \begin{equation}
\label{QN} Q_{G_N(\phi, \nu)}(\alpha) =
||z^{\alpha}||^2_{G_N(\phi, \nu)} = \int_{\C^m} |z^{\alpha}|^2
e^{- N \phi(z)} d\nu(z) \end{equation}
 of the monomials, viewed as homogeneous polynomials of degree $N$ (so that $|\alpha| \leq N$); equivalently, the Gaussian measures only
differ  in the variances of the coefficients in the monomial
basis. We refer to \S \ref{BACKGROUND} for details.

Since our emphasis is on the fewnomial aspects we only consider
some basic examples of $(\phi, \nu)$. In particular, we
concentrate on the model case of  $\SU(m + 1)$ polynomials,  where
$\phi(z) = \log (1 + ||z||^2)$ is the Fubini-Study potential and
where $d\nu = \frac 1{m!}(\frac{i}{2 \pi} \ddbar \phi)^m$ is the Fubini-Study
volume form.

Given the inner product $G_N(\phi, \nu)$  underlying the Gaussian
measure, we normalize the monomials to have $L^2$-norm equal to
one, by putting \begin{equation} \label{phialpha} \phi_{\alpha} =
\frac{z^{\alpha}}{Q_{G_N(\phi, \nu)}(\alpha)}, \end{equation} and
then express polynomials of degree $N$  as the orthonormal sums
\begin{equation} \label{PN} P_N = \sum_{\alpha \in N \Sigma} c_{\alpha} \phi_{\alpha}.
\end{equation}
The Gaussian measure $\gamma_N$ induced by $G_N(\phi, \nu)$ is
defined by the condition that the $c_{\alpha}$ are independent complex normal variables of mean zero and
variance one.

The Gaussian measure  $d\gamma_N $ on $\poly(N)$ induces
conditional Gaussian measures $\gamma_{N | S}$ on the spaces
$\poly(S)$; i.e.,
\begin{equation}\label{gammaNS} d\ga_{N|S}(P_N) = \frac 1{\pi^{|S|}}e^{-\sum|c_\al|^2}\,dc\,,\quad P_N=
\sum_{\alpha \in S} c_{\alpha} \phi_{\alpha}\,,
\end{equation} where $\phi_\al$ is given by \eqref{phialpha}.
Probabilities relative to $\gamma_{N|S}$  can be considered as
conditional probabilities; i.e., for any event $E$,
$$\mbox{Prob}_{\gamma_N}\{ P \in E | S_P = S\} = \mbox{Prob}_{\gamma_{N|S}} (E).$$
We denote by $\E_{N | S}$ the expectation with respect to the
conditional Gaussian measure $\gamma_{N | S}$.  For a further discussion of conditional probabilities on polynomial (and more general) ensembles, see \cite{SZZ}

Some of the possible (and well-studied) choices of the inner
products and Gaussian measures are the following:

\begin{itemize}

\item[(a)] The $\SU(m + 1)$ ensembles defined above. On all of $\poly(N)$
the expected distribution of zeros for each $N$ is uniform with
respect to the $\SU(m + 1)$-invariant volume form on $\CP^m$ (i.e.
the Fubini-Study form).

\item[(b)] General toric Gaussian measures induced by ${\mathbf T}^m$-invariant  Hermitian metrics $h = e^{- \phi}$ on the
line bundle $\ocal(1) \to \CP^m$ with positive curvature form
$\omega_{\phi} = i \ddbar \phi$ (i.e. with a plurisubharmonic
weight) and with $\nu = dV_{\phi} := \frac 1{m!}(\frac{i}{2 \pi} \ddbar \phi)^m$.
We suppress geometric notions in this article, but state the
general result in Theorem \ref{maintoric} (see \S \ref{TORIC}).

\item[(c)] The $m$-dimensional Kac-Hammersley ensembles, where  $Q_N(\alpha) = 1$. Here, $\phi
\equiv 0$ (the opposite extreme from subharmonic weights) and $\nu
= \delta_{{\mathbf T}^m}$. The norms of the monomials are independent of
$N$ and only involve a fixed inner product on $\C^m$. In dimension
one, the zeros of degree $N$ polynomials (with full spectrum $\Z\cap [0,N]$) concentrate on the unit circle   as $N\to \infty$ \cite{H}, and in dimension $m$ the zeros of degree $N$ polynomials concentrate on the torus ${\mathbf T}^m$ \cite{BS}. We briefly discuss this ensemble in \S \ref{KAC}.

\end{itemize}

\begin{rem}   Khovanskii \cite{KH} and
Kazarnovskii \cite{Ka1, Ka2} consider ensembles where one fixes the  spectra $S_j$  and chooses
coefficients at random from the ensemble
\begin{equation} \label{KHENS} \CP^{|S_1| - 1} \times \CP^{|S_2| -
1} \times \cdots \times \CP^{|S_m| - 1},\end{equation}  i.e., the
product projective space of coefficients of polynomials with the
prescribed spectra, equipped with the probability measure obtained by
taking the product of (normalized) Fubini-Study volume measures on
the factors. These ensembles amount to choosing the complex
coefficients at random from the Euclidean spheres $S^{2|S_j| -1}$
and are easily seen to be equivalent to Gaussian random polynomials $\sum_{\alpha} c_{\alpha}
z^{\alpha}$ with $c_{\alpha}$ independent complex normal variables
of mean zero and variance one---i.e., they are equivalent to the Kac-Hammersley ensembles described above.
\end{rem}

\subsection{Expected distribution of zeros}

Having fixed an ensemble of fewnomials, our interest is in the
configuration of zeros
$$Z_{P_1, \dots, P_k}:=\{z\in(\C^*)^m: P_1(z)=\cdots =P_k(z)=0\}$$
of a random fewnomial system with $k \leq m$. Here, $\C^*=\C\sm\{0\}$. We refer to $k = 1$
as the random fewnomial hypersurface case and to $k = m$ as the
point case.

 To each zero set we
 associate the  current of integration
 $\left[Z_{P_1, \dots,
 P_k}\right]\in \dcal'^{k,k}((\C^*)^m)$ over the zeros of the system:
 $$\left(\left[Z_{P_1, \dots,
 P_k}\right],\,\psi\right) = \int_{Z_{P_1, \dots,
 P_k}}\psi\,,\qquad \psi \in \dcal^{m-k,m-k}((\C^*)^m)\,.$$
  In the point case, $\left[Z_{P_1, \dots,
 P_m}\right] $  is obtained by putting point masses at each zero, $$\left[Z_{P_1, \dots,
 P_m}\right] = \sum_{z \in Z_{P_1, \dots, P_m}} \delta_z\,,$$
 and the expected distribution   is determined by
the expected values of the random variables
  $$\ncal^U_N(P_1^N,\dots,P_m^N):=\left[Z_{P_1, \dots,
 P_m}\right](U)=\#\{ z\in
U: P_1^N(z)= \cdots =P_m^N(z)=0\}$$ counting the number of
 zeros in an open set $U \subset (\C^*)^m$.
 
The expected distribution of zeros varies widely among the
ensembles above. This is not surprising if one recalls, for
instance, that zeros of random Kac polynomials concentrate on the
unit circle, while those of $\SU(2)$ polyonomials are uniform with
respect to the standard area form of $\CP^1 = \C \cup {\infty}$,
while those of polynomials with fixed Newton polytope have a
forbidden region where zeros have an exotic concentration. In
particular, the `average number' $S(P, U)$ of zeros in the angular
sector $U$ considered in Khovanskii's variance estimate
(\ref{KHOV}) is itself a random variable which depends on the
convex hull of the spectrum of $P$.

\subsection{Statement of results}

 We will consider the zero distribution as a measure on
 $(\C^*)^m = \R_+^m \times \T$ and denote points by $z = e^{\rho/2 + i \theta}$ in
 multi-index notation. Here, $\T$ denotes  the real torus $\T = (S^1)^m \subset
 (\C^*)^m$.
Given a locally bounded plurisubharmonic function $\phi$ we denote by
 $\ma(\phi)$ the associated Monge-Amp\`ere measure
 $$\ma(\phi) = \left(\frac i{2\pi}\ddbar \phi\right)^m \in \dcal'^{m,m}(\CP^m).  $$ When $\phi$ is
invariant under the $\T$  action  on $(\C^*)^m$,
 then
 \begin{equation} \label{MONGE} \ma(\phi) = \det\left(\frac 1{2\pi} D^2_{\rho} \phi \right)d\rho\, d \theta,  \end{equation}
  where $D^2_{\rho}$ is the real Hessian
 on $\R^m$.

 Our results are asymptotic formulas as the degree $N \to \infty$,
but with  the number $f$  of monomials held  fixed. For each
ensemble, the
 limit distribution of zeros in the point case is the
 Monge-Amp\`ere measure of a limit $\T$-independent potential, and
 thus the formula is of the type (\ref{MONGE}). The results are
 very similar for the Fubini-Study $\SU(m + 1)$ ensemble and for
 general toric Gaussian measures based on inner products
 $G_N(\phi, dV_{\phi})$ with $\phi$ a toric \kahler potential.
 Hence we concentrate on the $\SU(m + 1)$ case, and only briefly
 indicate the modifications needed for the general toric \kahler
 case.

 Our first result concerns the ensemble with dilates of a fixed
 spectrum. Since the lattice points lie in $\R^m$ we use upper
 subscripts to index the different points in the spectrum and
 lower subscripts to index their coordinates. We recall that
 $\E_{N | S}$ refers to the expectation with respect to the
 conditional Gaussian measure $\gamma_{N | S}$.

 \begin{maintheo}\label{EZa} Let $S=\{\la^1,\dots,\la^f\}$ be a fixed spectrum consisting of $f$ lattice points in $p\Si$. For random $m$-tuples $(P^N_1, \dots, P^N_m)$  of
fewnomials in $\poly(NS)$, with coefficients chosen from the $\SU(m + 1)$
ensembles of degree $pN$, the expected distribution of zeros in $(\C^*)^m$ has the asymptotics
$$N^{-m} {\bf E}_{Np|NS} [Z_{P^N_1, \dots, P^N_m}]\to  p^m\, \ma\left(
 \max_{\la\in S} \left[\langle \rho, \lambda
\rangle -\langle
\wh\lambda^p, \log \wh\lambda^p \rangle  \right]\right). $$ Here, $\wh \la^p= (p-|\la|,\la_1,\dots,\la_m)$ and $\log\wh\la^p=
(\log(p-|\la|),\log\la_1,\dots,\log\la_m)$.

\end{maintheo}

For a spectrum $S\subset p\Sigma$, we let $\lcal_S^p$ denote the Monge-Amp\`ere potential in Theorem \ref{EZa}:
\begin{equation}\label{Legendre} \lcal_S^p(\rho):= \max_{\la\in S}
 \left[\langle \rho, \lambda
\rangle -\langle
\wh\lambda^p, \log \wh\lambda^p \rangle  \right], \quad \rho\in\R^m\,.
\end{equation}
It is kind of discrete Legendre transform of the entropy function
$\langle \wh\lambda^p, \log \wh\lambda^p \rangle $, which is the
symplectic potential corresponding to the Fubini-Study \kahler
potential.

We note that the expected limit distribution is a singular measure
invariant under rotations of the angular variables and supported
along the 0-dimensional corner set  of the piecewise linear
function $ \lcal_S^p(\rho)$. This reflects the heuristic
principle that the zeros of a fewnomial should come from its
sub-fewnomials with fewnomial number $f=m+1$.

With no additional effort, we could fix the spectra separately for
each polynomial in the system, and obtain:

\begin{maintheo}\label{EZb} Let $S^1,\dots,S^k$ be fixed finite spectra consisting of lattice points in $p\Si$, where $1\le k\le m$. For random fewnomial $k$-tuples $(P^N_1, \dots, P^N_m)$ in $\poly(NS^1)\times\cdots\times\poly(NS^k)$, with coefficients chosen from the $\SU(m + 1)$
ensembles of degree $pN$, the expected zero current in $(\C^*)^m$ has the asymptotics
$$N^{-k} {\bf E}_{NS^1,\dots,NS^k} [Z_{P^N_1, \dots, P^N_k}]  \to    \bigwedge_{j=1}^k \left(\frac {ip}{2\pi}\ddbar\lcal_{S_j}^p(\rho)\right).$$
\end{maintheo}

We now state the result for fewnomial ensembles in which we
randomize the spectra in the sense of (II):

\begin{maincor}\label{EZc} Let $\Delta\subset p\Sigma$ be a (fixed) Newton polytope,
let $S^1,\dots,S^k$ be random spectra contained in $\De$ with fewnomial
number $f$, and let $P^N_1,\dots,P^N_k$
 be random fewnomial $k$-tuples $(P^N_1, \dots, P^N_m)$ in $\poly(NS^1)\times\cdots\times\poly(NS^k)$, with coefficients chosen from the $\SU(m + 1)$
ensembles of degree $pN$. Then the expected zero current in $(\C^*)^m$ has the asymptotics
$$N^{-k} {\bf E} [Z_{P^N_1, \dots, P^N_k}]  \to  \frac 1{C(\Delta,f)^k}  \left(\frac {ip}{2\pi}\ddbar\sum_{S\in\ccal(\Delta,f)}\lcal_{S}^p(\rho)\right)^k.$$

\end{maincor}

Next, instead of dilating random spectra, we  consider completely random spectra as described in (III) and we obtain:

\begin{maintheo}\label{main} Let $1\le k\le m$, and let $(P_1, \dots, P_k)$ be a random system of
fewnomials of fewnomial number $f$ and of degree $N$, where the spectra $S_j$ are chosen uniformly at random from
the simplex $N \Sigma$ and the coefficients are chosen from  the
$\SU(m + 1)$ ensemble. Then the expected zero current in
$(\C^*)^m$ has the asymptotics
$$N^{-k} {\bf E}_{N,f} [Z_{P^N_1, \dots, P^N_k}] \to    \left(\frac i{2\pi} \ddbar\int_{\Sigma^f}
 \max_{j = 1, \dots, f} \left[\langle \rho, \lambda^j
\rangle -\langle
\wh{\lambda^j\,}, \log \wh{\,\lambda^j\, } \rangle  \right]\, d \lambda^1 \cdots d \lambda^f\right)^k . $$
Here, $\wh \la=\wh\la^1= (1-|\la|,\la_1,\dots,\la_m)$, $d\la={m!} \,d\la_1\cdots d\la_m$.
\end{maintheo}

The limit measure is thus the Monge-Amp\`ere measure of the limit
potential obtained by averaging  the discrete Legendre transform $\lcal_{\{\lambda_1, \dots, \lambda_f\}}^1(\rho)$ from Theorems \ref{EZa} and \ref{EZb} (with $p=1$)  over all choices of  points $\la^1,\dots,\la^f$ of $\Sigma$. 

We note that the averaging smooths out the corners.  Indeed, we have the following more explicit formula for the expected limit distribution:

\begin{maincor}\label{maincor} Let $(P_1, \dots, P_k)$ be the random system of Theorem \ref{main}. Then 
$$N^{-k} {\bf E}_{N,f} [Z_{P^N_1, \dots, P^N_k}]
\to  \left\{\om_\FS - \frac i{2\pi}\int_0^\infty \ddbar\big(\left[1-D_b(t;\rho)\right]^f\big)\,dt\right\}^k \,,$$ where
$$D_b(t;\rho)=m!\,\vol\left(\{\la\in\Si:\langle
\wh\lambda, \log \wh\lambda \rangle - \langle \rho, \lambda
\rangle +\log\left(1+|e^\rho|\right)\le t\}\right)\,.$$
\end{maincor}

Here, $\om_\FS=\frac
i{2\pi}\ddbar \log (1+|e^\rho|)$ is the Fubini-Study \kahler form
on $(\C^*)^m\subset\CP^m$. The quantity  $D_b(t;\rho)$ is the distribution function for the pointwise logarithmic decay rate $b_\la(\rho)$ of the monomials $\phi_{N\la}$ (see \S \ref{MASS}), regarded as a random variable (with parameter $\rho$) on $\Sigma$.  Note that the integral in Corollary \ref{maincor} is actually over a bounded interval.

We can also generalize Theorem \ref{main} to the ensemble (IV):

\begin{maintheo}\label{mainIV} Let $1\le k\le m$, let $\Delta_1\subset p_1\Sigma,\dots,\Delta_k\subset p_k\Sigma$ be Newton polytopes and let $(P_1, \dots, P_k)$ be a random system of
fewnomials of fewnomial numbers $f_1,\dots,f_k$ respectively, where the spectra $S_j$ are chosen uniformly at random from
the simplices $N \Delta_j$ and the coefficients are chosen from  the
$\SU(m + 1)$ ensemble. Then the expected zero current in
$(\C^*)^m$ has the asymptotics
$$N^{-m} {\bf E} [Z_{P^N_1, \dots, P^N_k}] \to   \bigwedge_{j=1}^k \left(\frac i{2\pi}\,\frac{p_j}{\vol(\Delta_j)^{f_j}}\, \ddbar\int_{\Delta_j^{f_j}}
 \max_{l = 1, \dots, f_j} \left[\langle \rho, \lambda^l
\rangle -\langle
\wh{\lambda^l\,}, \log \wh{\,\lambda^l\, } \rangle  \right]\, d \lambda^1 \cdots d \lambda^{f_j}\right) . $$
Here, $\wh \la=\wh\la^1= (1-|\la|,\la_1,\dots,\la_m)$, $d\la= \,d\la_1\cdots d\la_m$.
\end{maintheo}

The key analytical ingredient in the proofs of these results is an
asymptotic formula for  the expected {\it mass density} of the
above systems of random polynomials as $N \to \infty$. It is given
by the conditional \szego kernel $\Pi_{N, Q | S}$ with respect to the
norm $Q_{G_N(\phi,\nu)}$ (\ref{QN}) and spectrum $S$, i.e. the
kernel of the
 orthogonal projection (\szego kernel) onto the subspace of
polynomials under consideration:
\begin{equation}\label{Eszego}
\E_{N,Q |S}\left(|P(z)|^2_{N\phi}\right) = \sum_{\alpha \in S} \frac{
|\chi_\al(z)|_{N\phi}^2}{\|\chi_{\alpha}\|_Q^2}=\Pi_{N, Q|S
}(z,z)\;.\end{equation} Thus, the results depend on the
asymptotics of the \szego kernels $\Pi_{N, Q |S }(z,z)$.

\subsection{More general toric weights}

We briefly indicate the generalization when the $\SU(m + 1)$
(Fubini-Study) inner product on $\poly(N)$ is replaced by
$G_N(\phi, dV_{\phi})$ for a general toric \kahler potential
$\phi$.

The polytope $P$ of the toric variety is
 defined by a set of linear inequalities
$$\ell_r(x): =\langle x, v_r\rangle-\lambda_r \geq 0, ~~~r=1, ..., d, $$
where $v_r$ is a primitive element of the lattice and
inward-pointing normal to the $r$-th $(m-1)$-dimensional facet
$F_r = \{\ell_r = 0\}$  of $P$.

A     $\T$-invariant \kahler potential on $(\C^*)^m$  defines a
real convex function on $\rho \in \R^m$. Its Legendre transform
$$u_{\phi}(x): = \lcal \phi (x): = \sup_{\rho} \left( \langle x,
\rho \rangle - \phi(e^{\rho}) \right) $$ is the {\it symplectic
potential} $u_{\phi}$. Equivalently, for  $x \in P$,  there is a
unique $\rho$ such that $ \nabla_{\rho} \phi = x$, and
$u_{\phi}(x) = \langle x, \rho_x \rangle - \phi(\rho_x)$. In the
Fubini-Study case, $P = \Sigma$,  $\phi = \log (1 + e^{\rho})$,
and $u_{FS}(x) =  \sum_k \ell_k(x) \log \ell_k(x)$ where
$\ell_k(x) = x_k$ for $k = 1, \dots, m $ and $\ell_{m + 1}(x) = 1
- |x|$ where $|x| = x_1 + \cdots + x_m$  (in multi-index notation
on $\R^m$). Thus,
$$u_{FS}(\lambda) = \langle \wh\lambda, \log \wh\lambda \rangle.
$$

The \kahler potential is the Legendre transform $\lcal
u_{\phi}(\rho)$ of its symplectic potential. If we allowed all
possible spectra in the ensemble (hence not a fewnomial ensemble),
the discrete Legendre transforms with respect to $f$-element
subsets would converge to the usual Legendre transform and the
potential in (\ref{Legendre}) would become $\phi$. Thus, the
impact of the restriction to $f$ monomials is that in place of the
Legendre transform we have an average of discrete Legendre
transforms.

As this indicates, the result for a general toric \kahler Gaussian
ensemble for $\CP^m$ and polytope $\Sigma$,  defined by $G_N(\phi,
dV_{\phi})$, is the following:

\begin{maintheo}\label{maintoric} Consider the ensembles of type {\rm (III)} as in Theorem \ref{main},
but with Gaussian measures induced by the inner product $G_N(\phi,
dV_{\phi})$ corresponding to a toric \kahler potential on $\CP^m$.
Then the expected distribution of zeros in $(\C^*)^m$ has the
asymptotics
$$N^{-m} {\bf E} [Z_{P^N_1, \dots, P^N_k}] \to    \left(\frac i{2\pi} \ddbar\int_{\Sigma^f}
 \max_{j = 1, \dots, f} \left[\langle \rho, \lambda^j
\rangle - u_{\phi}({\lambda}) \right]\, d \lambda^1 \cdots d
\lambda^f\right)^k .
$$
\end{maintheo}

The proof is almost the same as for the Fubini-Study case and is
indicated in \S \ref{TORIC}. In \S \ref{KAC}, we also indicate the
modifications in the case of the fewnomial Kac-Hammersley ensemble.

\section{Preliminaries}\label{BACKGROUND}

In this section, we review the relation between inner products on
spaces of polynomials and associated Gaussian measures on the
space. The inner products implicitly involve a choice of \kahler
metric on $\CP^m$. The associated \kahler potential determines the
shape of the modulus of each monomial and its concentration
properties.

 We may identify a not necessarily homogeneous polynomial $f$ on $\C^m$ of
degree $\leq N$ by its homogenization as a polynomial
$$F(\zeta_0, \dots,
\zeta_m) = \sum_{|\la| = N} C_\la \zeta^\la \qquad (\zeta^\la =
\zeta_0^{\la_0} \cdots \zeta_m^{\la_m})$$ of degree $N$ in $m+1$
variables, where
$$f(z_1,\dots,z_m)=F(1,z_1,\dots,z_m) = \sum _{|\al|\le N} c_\al
z^\al \qquad (z^\al=z_1^{\al_1} \cdots z_m^{\al_m}),$$ where
$c_\al=C_{\hat\al^N}$, $\hat\al^N=(N-|\al|,\al_1,\dots,\al_m)$,
$|\al|= \sum_{j=1}^m\al_j$. Homogeneous polynomials of degree $N$
on $\C^{m + 1}$ are equivalent to holomorphic sections $H^0(\CP^m,
\ocal(N))$ of the $N$th power of the hyperplane section bundle.
This geometric identification is useful in interpreting the
concentration properties of monomials in terms of curvature.

We let $e_0\in H^0(\CP^m\ocal(1))$ be the degree 1 polynomial $e_0(\zeta_0,\dots,\zeta_m)=\zeta_0$. Then $e_0$ is a local
frame  over the affine chart $U_0 = \{\zeta_0\neq 0\}\approx\C^m$.
We fix a Hermitian metric $h$ on $\ocal(1)$. In the local frame $e_0$, the
metric has the local expression $h = e^{- \phi}$, where $\phi$ is
known as the \kahler potential.  The \kahler form is denoted
by $\omega_{\phi} = \frac i{2\pi}\, \ddbar \phi$.

We  define the inner product on
$\poly(N\Sigma)$:
\begin{equation}\label{IP}\langle f, \bar g \rangle_h =   \frac{1}{m!} \int
_{\C^m} f(z)\overline{g(z)} e^{- N \phi(z)}
\,\om_{\phi}^m(z),\quad f,g\in \poly(N \Sigma). \end{equation}

The inner product is determined by the matrix of inner products on
the distinguished basis of monomials $\chi_{\alpha}$. All of our
inner products are $\T$-invariant and hence the monomials are
automatically orthogonal. The inner products are then determined
by the norming constants (\ref{QN}), specifically,
\begin{equation}\label{NORM} Q(\al)=Q_{G_N(\phi, dV_{\phi})} (\alpha)  =   \frac{1}{m!} \int
_{\C^m} |z^{\alpha}|^2  e^{- N \phi(z)} \,\om_{\phi}^m(z).
\end{equation}

The inner product induces a Gaussian measure $\gamma_h$ on any
subspace $\scal \subset  \poly(N \Sigma)$. Again assuming that the
monomials are orthogonal, the basis (\ref{phialpha}) is $\langle,
\rangle_h$ orthonormal and we may write any polynomial in the form
$$P_N = \sum_{\alpha \in N \Sigma} c_{\alpha} \phi_{\alpha}. $$
The associated Gaussian measure is defined by the condition that
the coefficients of this orthonormal expansion are independent
complex normal random variables.

The \szego kernel (or weighted Bergman kernel) for the line bundle $\ocal(N)$ with metric $h^N=e^{-N\phi}$  is given over $\C^m$ by
\begin{equation}\Pi_{N,Q}(z,w) = e^{-N\phi}\sum_{\alpha \in N \Sigma}\phi_{\alpha}^N(z)  \overline{\phi_{\alpha}^N(w)}\;.\end{equation}  
It is the kernel for the orthogonal projection from $L^2(X)\to H^0(\CP^m,\ocal(N))$, where $X\to\CP^m$ is the unit circle bundle in $(L^*,h^*)$ with fibers $X_z=  \{e^{-\phi+i\theta}(1,z_0,\dots,z_m):\theta\in\R\}$ over points $z\in\C^m$; see \cite{SZ}.
For spectra $S\subset\Z^m\cap N\Sigma$, then the  kernel for the orthogonal projection $L^2(X)\to \poly(S)$ is the {\it conditional weighted Bergman kernel\/} given by
\begin{equation}\label{szego}\Pi_{N,Q|S}(z,w ) = e^{-N\phi}\sum_{\alpha \in S}
\phi_{\alpha}^N(z)  \overline{\phi_{\alpha}^N(w)}\;.\end{equation}

\subsection{The $\SU(m+1)$-ensembles}

 This is the Gaussian ensemble defined by the inner product
 arising from the
 Fubini-Study metric $\phi = \log (1 + |z|^2)$. Then  $\om_\FS=\frac
i{2\pi}\ddbar \log (1+\|z\|^2)$ is the Fubini-Study \kahler form
on $\C^m\subset\CP^m$ and
\begin{equation}\label{IPa}\langle f, \bar g \rangle =   \frac{1}{m!} \int
_{\C^m}\frac{f(z)\overline{g(z)}}{(1+\|z\|^2)^N}
\,\om_\FS^m(z),\quad f,g\in  \poly(N \Sigma),.\end{equation}

 The norming constants for the
inner product (\ref{IP}) are:
\begin{equation}\label{IP2}\|\chi_\al\| =\sqrt{ \langle \chi_\al, \chi_\al\rangle} =
\left[\frac{N!}{(N+m)!{N\choose\al}}\right]^\half\;,\quad\quad
{N\choose\al}:= \frac{N!}{(N-|\al|)!\alpha_1!\cdots \alpha_m!}\;.
\end{equation}
Thus we have an orthonormal basis for $\poly(N \Sigma)$ given by
the monomials
\begin{equation}\label{normchi}m^N_\al:= \frac{1}{\|\chi_\al\|}\,\chi_\al=
\sqrt{\frac{(N+m)!}{N!} {N\choose \al}}\ \chi_\al\ ,\qquad
|\al|\le N\;.\end{equation}

In this case, the circle bundle $X$ is the unit sphere $S^{2m+1}\subset \C^{m+1}$. We now regard the sections of $H^0(\CP^m,\ocal(N))$ as homogeneous polynomials restricted to $X=S^{2m+1}$.  By identifying the point $z\in(\C^*)^{m}$ with
the lift $x=\frac 1{(1+\|z\|^2)^{1/2}}(1,z_1,\dots,z_m)\in
S^{2m+1}$,  we may write the homogenized monomials on $S^{2m +
1}$ in affine coordinates $(z_1, \dots, z_m)$ as
\begin{equation}\label{mchi}\wh
\chi^N_\al(z)= \frac{z^\al}{(1+\|z\|^2)^{N/2}}.\end{equation} The
corresponding $L^2$ normalized monomials are then:
\begin{equation}\label{normchi1} \wh m^N_{\alpha}(z): =
\sqrt{\frac{(N+m)!}{N!} {N\choose \al}}\
\frac{z^\al}{(1+\|z\|^2)^{N/2}}\ ,\qquad |\al|\le N\;.
\end{equation}  

In short,  the $\SU(m + 1)$ ensemble of random polynomials of degree $N$
consists of polynomials of the form
\begin{equation}\label{normchi2} \sum_{|\alpha| \leq N} c_{\alpha}
\sqrt{\frac{(N+m)!}{N!} {N\choose \al}}\
\frac{z^\al}{(1+\|z\|^2)^{N/2}}\ ,
\end{equation}
where $c_{\alpha}$ are independent complex normal variables of mean zero and
variance one.

Specializing \eqref{szego} to the Fubini-Study metric, we have the following definition (where we omit the
subscript $Q$ indicating the norming constants):

\begin{defin}  Let $S\subset \Z^m\cap N\Sigma$. The  conditional Fubini-Study \szego kernel $\Pi_{N |S}$ is the kernel for the orthogonal
projection to $\poly(S)$ with respect to the induced Fubini-Study
inner product:
\begin{equation}\label{Sz2} \Pi_{N|S}(x,y) = \sum_{\alpha \in S}
\frac{1}{\|\chi_\al\|^2} \wh\chi^N_\al(x) \overline{\wh
\chi^N_{\alpha}(y)}= \frac{(N+m)!}{N!} \sum_{\al\in P} \wh m
^N_\al(x) \overline{\wh m^N_{\alpha}(y)}
\;.\end{equation}\end{defin}
  The
conditional \szego kernel can be written explicitly on $\C^m$ as
\begin{equation}\label{Sz2e} \Pi_{N|S}(z,w) =\frac{(N+m)!}{N!} \frac{\sum_{\alpha
\in S}{N\choose\al}z^\al \bar w^\al}{(1+\|z\|^2)^{N/2}
(1+\|w\|^2)^{N/2}} \;.\end{equation} It is the two-point function
for the conditional Gaussian ensemble $\poly(S)\subset\poly(N)$.

The full Fubini-Study  \szego kernel is given by
\begin{eqnarray}\Pi_{N}(z,w) &=& \frac{(N+m)!}{N!}
\sum_{|\alpha|\le N}\wh \chi^N_\al(z) \overline{\wh \chi^N_\al(w)}
\ =\ \frac{(N+m)!}{N!} \frac{\sum_{|\alpha|\le N
}{N\choose\al}z^\al \bar w^\al}{(1+\|z\|^2)^{N/2}
(1+\|w\|^2)^{N/2}} \label{Sz} \\&=& \frac{(N+m)!}{N!}
\left[\frac{1+\langle z,\bar w\rangle} {(1+\|z\|^2)^{1/2}
(1+\|w\|^2)^{1/2}}\right]^N\;.\label{Sz1}
\end{eqnarray}

\section{Fewnomial Ensembles}\label{FEWENSEMBLES}

\subsection{Precise definitions of random fewnomials}\label{ENSEMBLES}

We now define more precisely the  ensembles which allow for any
fewnomial system.  We fix the degree $N$, and first consider the
case of one random fewnomial. We specify a set of lattice points
by its characteristic function
\begin{equation} \sigma: N \Sigma \cap \N^m \to \{0, 1\},
\end{equation}
which may be regarded as an  occupation number, designating
whether a lattice point is occupied ($\sigma(\alpha) = 1)$ or
unoccupied ($\sigma(\alpha) = 0$. We denote by $|\sigma| =
\sum_{\alpha \in N \Sigma} \sigma(\alpha)$ the number of elements
in the set, and by $\supp \sigma = \{\alpha: \sigma(\alpha) = 1\}$
the support of $\sigma.$  We put:
\begin{equation} {\mathcal C}_{N, f} = \{ \sigma: N \Sigma \cap \N^m \to
\{0, 1\}\;\; \mbox{such that}\;\; |\sigma| = f \}, \end{equation}
and we denote the number of such subsets by
\begin{equation} \label{CNf}C(N, f) = |{\mathcal C}_{N, f}| = {{N+m\choose m}\choose f} = \frac 1{(m!)^f\,f!}\,N^{mf} +O(N^{mf-1}). \end{equation}
A polynomial with (at most) $f$ non-zero terms can then be written
in the form:
\begin{equation} P_{\sigma, c}(z) = \sum_{\alpha \in N \Sigma}
\sigma(\alpha)\, c_{\alpha}\, z^{\alpha},\quad |\sigma|=f\,. \end{equation}
 Thus the space of
random fewnomials is given by:
\begin{equation} {\mathcal F}_{N, f}  = \{(\sigma, P)
\in {\mathcal C}_{N, f} \times
\poly(N): P \in \poly(\supp \sigma) \}.
\end{equation}
There is a natural projection $\pi: {\mathcal F}_{N, f} \to
{\mathcal C}_{N, f}$ taking $(\sigma, P) \to \sigma$ and the
`fiber' of this projection is $\poly(\supp \sigma).$
 The
set of fewnomial systems of $m$ polynomials in $m$ variables with
fewnomial numbers $(f_1, \dots, f_m)$ is then given by
\begin{equation} {\mathcal F}_{N, (f_1, \dots, f_m)}:= {\mathcal F}_{N, f_1} \times \cdots \times
{\mathcal F}_{N, f_m}. \end{equation}

It is also natural to consider fewnomials with spectra contained
in a given Newton polytope. We therefore fix a convex lattice
polytope $\Delta \subset p \Sigma$ (for some $p$) and replace
$\Sigma$ everywhere by $\Delta.$ Thus, we  define
\begin{equation} {\mathcal C}_{N, f, \Delta} = \{ \sigma: N \Delta \cap \N^m \to
\{0, 1\}\;\; \mbox{such that}\;\; |\sigma| = f \},
\end{equation}
and
\begin{equation} {\mathcal F}_{N, f, \Delta} \subset {\mathcal C}_{N, f, \Delta} \times
\poly(N) = \{(\sigma, P): \supp(\sigma) \subset N \Delta,\; P \in
\poly(\supp \sigma) \}..
\end{equation}
Similarly, we define ${\mathcal F}_{N, (f_1, \dots, f_m),
\Delta_1, \dots, \Delta_m}$ for systems.

We now induce probability measures on ${\mathcal F}_{N, f}$ and
${\mathcal F}_{N, f, \Delta},$ by regarding them as `fibering'
over ${\mathcal C}(N, f)$, by putting counting measure on
${\mathcal C}_{N, f}$ and by putting the conditional measures $d
\gamma_{N| supp \sigma}$ on the `fibers'.

\begin{defin}\label{SUFDEF} The ensemble of random $\SU(m + 1)$ fewnomials of
degree $N$ and fewnomial number $f$ is the space ${\mathcal F}_{N,
f}$ endowed with the probability measure $d \mu_{N, f}$ defined by
$$ \int_{{\mathcal F}_{N, f}} g(S, P)\, d\mu_{N, f}(S,P) := \frac{1}{C(N,
f)} \sum_{S \in {\mathcal C}(N, f)} \int_{\poly(S)} g(S, p)\,
d\gamma_{N| S}(P).
$$
\end{defin}

In other words, $d\mu_{N, f}$ is defined by putting counting
measure on ${\mathcal C}_{N, f}$ and by putting the conditional
measures $d \gamma_{N| S}$ (given by \eqref{gammaNS} with $\phi_\al=\wh m_\al^N$) on the `fibers' of $\pi$.

We then put the product measures $$d\mu_{N, f_1, \dots, f_k} =
d\mu_{N, f_1} \times \cdots \times d\mu_{N, f_k}$$ on the space
${\mathcal F}_{N, (f_1, \dots, f_m)}$ of systems.

We define the measure $d \mu_{N, f, \Delta}$  on ${\mathcal F}_{N,
f, \Delta}$ and on the associated systems analogously. Similarly we
define the measures $d\mu_{N, f}^{\phi,\nu}$ and  $d\mu_{N, f}^{KH}$ for the general toric and Kac-Hammersley ensembles, respectively.

\subsection{Expected zeros of fewnomial ensembles}

We recall the {\it probabilistic Poincar\'e-Lelong formula\/} (see
for example,\cite{SZ,SZ3}):

\begin{prop}\label{review} Let $(L,h)$ be a Hermitian line bundle on a
compact \kahler manifold $M$. Let $\scal$ be a 
subspace of $ H^0(M,L)$ endowed with a Hermitian inner
product and we let $\ga$ be the induced  Gaussian probability measure
on $\scal$. Then the expected zero current of a random section
$s\in\scal$ is given by
\begin{eqnarray*}\E_\ga(Z_s)  &=&\frac{\sqrt{-1}}{2\pi}
\partial
\bar{\partial} \log \Pi_{\scal}(z, z)+c_1(L,h)\;.\end{eqnarray*}

If $\scal_j$ is a base-point-free linear system with Gaussian probability measure $\ga_j$, for $1\le j\le k$ (where $1\le k\le m$), then the  expected value of the simultaneous zero current of  $k$ independent
random sections $s_1\in\scal_1,\dots,s_k\in\scal_k$  is given by
$$\E_{\ga_1,\dots,\ga_k}\big(Z_{s_1,\dots,s_k}\big) = \bigwedge_{j=1}^k\left(\frac{\sqrt{-1}}{2\pi}
\partial
\bar{\partial} \log \Pi_{\scal_j}(z, z)+c_1(L,h)\right)\;,$$
which is a smooth form.\end{prop}

Applying Proposition \ref{review} to a fewnomial system $\scal=\poly(S)$, we have
\begin{prop}\label{review1} Let $S_1,\dots,S_k$ be finite subsets of $ N\Sigma \cap \Z^m$. Then
the expected zero current in $(\C^*)^m$ of $k$ random fewnomials
$P_1\in\poly(S_1),\dots,P_k\in S_k$ is given by the smooth form
$${\bf E}_{N|S_1,\dots,S_k} Z_{P_1,\dots,P_k} =  \bigwedge_{j=1}^k\left(\frac{\sqrt{-1}}{2\pi}
\partial
\bar{\partial} \log \Pi_{S_j}(z, z)+\frac N\pi\om_\FS\right)\;.$$

 \end{prop}

\begin{proof} We recall that the  base point
locus of a suspace ${\mathcal S} \subset \poly(N)$ is the set of
points at which $p(z) = 0, \forall p \in {\mathcal S}$.Since a monomial
$z_1^{\alpha_1} \cdots z_m^{\alpha_m}$ vanishes if and only if
$z_j = 0$ for some $j$ such that $\alpha_j > 0$, the base locus of
$\poly(S)$ is always contained in the coordinate hyperplances $\bigcup_{j = 1}^m \{z_j =
0\}.$  Applying Proposition \ref{review} to $(\C^*)^m\subset \CP^m$, we obtain the result.
\end{proof}

\begin{cor} \label{review2}
The expected zero current in $(\C^*)^m$ of a system of $k$ random fewnomials of degree $\le N$ with fewnomial number f is given by
$${\bf E}_{N,f} Z_{P_1,\dots,P_k} =\left[ \frac{1}{C(N, f)} \sum_{\sigma \in {\mathcal C}_{N,f}} \left(\frac{\sqrt{-1}}{2\pi}
\partial
\bar{\partial} \log \Pi_{N|\supp\sigma}(z, z)+\frac N\pi\om_\FS\right)\right]^k\,.$$ where $C(N,f)$ is given by \eqref{CNf}.
 \end{cor}

\subsection{Mass asymptotics and fewnomial \szego kernels}\label{MASS}

We now give the asymptotics of the \szego kernels $\Pi_{N|S}$ We need joint asymptotics in $N$ and $S$ (leaving the fewnomial number $f=|S|$ fixed). We begin with the dilated fixed spectra system (I).

A special case of Theorem 4.1 in \cite{SZ} on the mass asymptotics for polynomials with spectra in dilates of a Newton polytope $P$ is where  $P=\{\be\}$ is a single lattice point in
$p\Si$. In this case
\begin{equation}\label{Pidecay} \Pi_{Np|N\{\be\}}(z,z)= |\wh m^{Np}_{N\be}|^2= N^{\frac m2}
e^{-Nb_\be^p(z)}[c_0+c_1 N\inv +c_2 N^{-2}+ \cdots]\;,\end{equation}
where
\begin{equation}\label{bpoint} b_\be^p(z)= \sum_{j=0}^m\be_j\log \frac{\be_j}{p}
-\log\frac{|z^\be|^2}{(1+\|z\|^2)^p}\qquad
(\be_0=p-|\be|)\;.\end{equation}  In \eqref{bpoint}, we can let $\be$
be any point in the interior of $p\Sigma$.  We also write $b_x=b_x^1$, for arbitrary (not necessarily integral) $x\in\Si$:
\begin{equation}\label{bx}  b_x(z)= \sum_{j=0}^mx_j\log x_j - \sum_{j=1}^m x_j\log |z_j|^2+\log{(1+\|z\|^2)}\qquad
(x_0=1-|x|)\;.\end{equation}

The first term is the symplectic potential for the Fubini-Study
metric, i.e. the Legendre transform of the open orbit \kahler
potential (see \S \ref{TORIC})  We now give a precise estimate for
the joint asymptotics of \eqref{Pidecay} using Stirling's formula.
A similar analysis was done in dimension one in \cite{SoZ1} and in
Lemma 6.2 of \cite{SoZ3}. The kernel (\ref{Pidecay}) is denoted
$\pcal_{h^N}(\alpha, z)$ in \cite{SoZ} and is analyzed for general
toric varieties in Section 6 of that article.  Since it is
elementary we give a self-contained proof in the case of $\SU(m + 1)$
polynomials (i.e., for the Fubini-Study metric).

\begin{lem}\label{ptwise} There exists positive constants $C_m$ depending only on $m$ such that for all $\al\in (N\Si)^\circ\cap\Z^m$, we have
$$ \log |\wh m^{N}_{\al}|^2=  -N\,b_{\al/N}+\frac m2 \log N -\half\sum_{j=0}^m\log\left(\frac{\al_j}N\right)+R(\al,N,m)\,,$$
where $|R(\al,N,m)| \le C_m$.

\end{lem}
\begin{proof} Let $x=\al/N$.  Recalling \eqref{normchi1}, it suffices to show that
$$\log\left[\frac{(N+m)!}{N!} {N\choose \al}\right] = -\sum_{j=0}^m \left(Nx_j+\frac 12\right)\log x_j+\frac m2 \log N +R(\al,N,m)\qquad (\al_0 = N-|\al|)\,.$$
Using Stirling's formula
\begin{equation}\label{Stirling} n! = \sqrt{2\pi} \,n^{n+ 1/2}\, e^{-n+\ep_n}\,, \quad \mbox{where }\ \frac 1{12n+1} <\ep_n< \frac 1{12n}\;,\end{equation}
we obtain \begin{eqnarray*}\log\left[\frac{(N+m)!}{N!} {N\choose \al}\right]\hspace{-1.3in} \\&=&\sum _{j=1}^m\log(N+j) - \frac m2 \log (2\pi) +(N+\half)\log N-\sum_{j=0}^m(\al_j+\half)\log\al_j +\ep_N - \sum_{j=0}^m\ep_{\al_j}\\&=&\frac m2 \log N-\sum_{j=0}^m \left(Nx_j+\frac 12\right)\log x_j +R\,,\end{eqnarray*} where $$R= \sum_{j=1}^m \log (1+j/N)-\frac m2\log(2\pi) +\ep_N- \sum\ep_{\al_j}.$$ Thus
$$|R|\le \sum_{j=1}^m \log (1+j) +\frac m2\log(2\pi) +\frac{m+1}{12}\,.$$
\end{proof}

\begin{lem}\label{ptwisebound} There exists positive constants $C'_m$ such that
$$ -N\,b_{\al/N}+\frac m2 \log N -C'_m\le \log |\wh m^{N}_{\al}|^2\le  -N\,b_{\al/N}+m\log N +C'_m\,,$$ for all $\al\in (N\Si)\cap\Z^m$.\end{lem}
\begin{proof}  We first suppose that $\al\in I_N:=(N\Sigma)^\circ\cap\Z^m$. The lower bound is an immediate consequence of Lemma \ref{ptwise}. If $I_N\neq\emptyset$, then
 $N\ge m+1$ and the maximum value of the convex function $\alpha\mapsto-\sum_{j=0}^m\log\left(\frac{\al_j}N\right)$ on $I_N$  is attained on the vertices of $I_N$. Thus $$-\half\sum_{j=0}^m\log\left(\frac{\al_j}N\right)\le \frac m2 \log N +\half \log \frac N{N-m} \le  \frac m2 \log N +\half\log(m+1)\,,$$ and the upper bound follows from Lemma \ref{ptwise}.

Now suppose that $\al\in\d(N\Sigma)\cap\Z^m$. By a permutation of homogenous coordinates, we can assume without loss of generality that $\al=(\al_0,\dots,\al_k,0,\dots,0)$ where $\al_j\ge 1$ for $0\le j\le k$.  Let $\al'=(\al_1,\dots,\al_k)$, $z'=(z_1,\dots,z_k)$.  We note that
$$ b_{\al'/N}(z')-b_{\al/N}(z) =
\log(1+\|z'\|^2) - \log(1+\|z\|^2)\,.$$

By the lower bound proved above for the monomial $\wh m^N_{\al'}$ on $(\C^*)^k$, we have

\begin{eqnarray*}\log|\wh m^N_{\al}(z)|^2 &=&\log|\wh m^N_{\al'}(z')|^2 +\log\left[\frac{(N+m)!}{(N+k)!}\right] + N\log\left[\frac{1+\|z'\|^2} {1+\|z\|^2}\right]\\& \ge & - N\,b_{\al'/N}(z') +\frac k2 \log N - C'_k +(m-k)\log N + N\log\left[\frac{1+\|z'\|^2} {1+\|z\|^2}\right]\\
&=& - N\,b_{\al/N}(z) +\left(m-\frac k2\right)\log N -C_k',\end{eqnarray*} which yields the desired lower bound when $\al$ is in the boundary of $N\Sigma$.

On the other hand, by the upper bound for the monomial $\wh m^N_{\al'}$, we have
\begin{eqnarray*}\log|\wh m^N_{\al}(z)|^2 &=&\log|\wh m^N_{\al'}(z')|^2 +\log\left[\frac{(N+m)!}{(N+k)!}\right] + N\log\left[\frac{1+\|z'\|^2} {1+\|z\|^2}\right]\\&\le& - N\,b_{\al'/N}(z') +k\log N +C_k'+(m-k)\left(\log N +\frac mN\right) + N\log\left[\frac{1+\|z'\|^2} {1+\|z\|^2}\right]\\
&=& - N\,b_{\al/N}(z) +m\log N +C_k'+\frac{m^2-km}N\,,\end{eqnarray*} which gives the desired upper bound.
\end{proof}

\subsection{Proof of Theorems \ref{EZa} and \ref{EZb}} These theorems are consequences of the following convergence result:

\begin{lem}\label{uniform} Let $m,f,p$ be positive integers.  Then $$\frac 1N \log \Pi_{Np|NS}(z,z) \to -p\,\min_{1\le j\le f} \{b_{\la^j/p}(z)\}$$ uniformly for $z\in(\C^*)^m,\ S\in\ccal(p,f)$.
\end{lem}
\begin{proof}
Let $S=\{\la^1,\dots,\la^f\}\subset p\Sigma$, and recall that
\begin{equation*} \Pi_{Np|NS}(z,z) = \sum_{j=1}^f |\wh m^{Np}_{N\la^j}(z)|^2 \,.
\end{equation*}  By Lemma \ref{ptwisebound}, we have
\begin{multline*}\max_j \{-Np\,b_{\la^j/p}\} -C_m \le \max_j \{\log |\wh m^{Np}_{N\la^j}|^2\} \le \log  \Pi_{Np|NS} \le \max_j \{\log |\wh m^{Np}_{N\la^j}|^2\} +\log f\\
\le  \max_j\{-Np\,b_{\la^j/N}\}+m\log (Np) +C'_m+\log f\,.\end{multline*}
Dividing by $N$, the conclusion follows.\end{proof}

\begin{rem}  Lemma \ref{uniform} is a special case of the generalization (with a stronger
 uniformity result) of Proposition 4.2 in \cite{SZ2} to nonconvex polytopes [unpublished]. In
 the case where $S$ is one point, an  analysis of the full (i.e. not just logarithmic)  asymptotics of $\pcal_{h^N}(\alpha, z)$
is given  in Section 6 of \cite{SoZ}.  \end{rem}

\medskip\noindent{\it Proof of Theorems \ref{EZa}--\ref{EZb}:\/}Let $S=\{\la^1,\dots,\la^f\}\subset p\Sigma$.
 By Proposition \ref{review1} with $k=m$,
\begin{eqnarray*}N^{-m}\E_{N|NS}Z_{p_1,\dots,p_m}
&=&\left\{\frac i{2\pi}\ddbar\left[\frac 1N \log \Pi_{Np|NS}(z,z)\right]+\frac p\pi\, \om_\FS\right\}^m
\\&=&\ma\left\{\left[\frac 1N \log \Pi_{Np|NS}(z,z)\right]+p\log(1+\|z\|^2)\right\}.\end{eqnarray*}
 By \eqref{bx} and Lemma \ref{uniform}, \begin{equation}\label{limI}\left[\frac 1N \log \Pi_{Np|NS}(z,z)
 \right]+p\log(1+\|z\|^2)\to p\,\max_{\la\in S}\left[\langle\rho, \lambda^p
\rangle -\langle
\wh\lambda^p, \log \wh\lambda^p \rangle  \right] +p\log p\end{equation} uniformly, where
 $\rho=(\log|z_1|^2,\dots,\log|z_m|^2)$. Theorem \ref{EZb} then follows from Proposition \ref{review1}
  and the Bedford-Taylor theorem \cite{BT,Kl} on the continuity of the operator $(u_1,\dots,u_k)\mapsto dd^cu_1\wedge\cdots\wedge dd^c u_k$ under uniform limits.
Theorem \ref{EZa} is a special case of Theorem \ref{EZb}.\qed

\medskip Corollary\ref{EZc} follows immediately from Theorem \ref{EZb} by averaging over the spectra in $\Delta$.

\section{Zeros of random fewnomial systems: Proof of Theorem \ref{main}}\label{proofmain}

By Corollary \ref{review2} and the Bedford-Taylor continuity theorem for $dd^cu_1\wedge\cdots\wedge dd^c u_k$ under uniform limits, to prove Theorem \ref{main}, it suffices to show that
\begin{multline}\label{sumlim}
\frac{1}{C(N, f)} \sum_{\sigma \in {\mathcal C}_{N,f}} \left(\frac 1N \log \Pi_{N|\supp\sigma}(z, z)+ \log(1+\|z\|^2)\right)\\ \to \int_{\Sigma^f}
 \max_{j = 1, \dots, f} \left[\langle \rho, \lambda^j
\rangle -\langle
\wh{\lambda^j\,}, \log \wh{\,\lambda^j\, } \rangle  \right]\, d \lambda^1 \cdots d \lambda^f\end{multline} uniformly on compact subsets of $(\C^*)^m$.

We begin by writing the above sum as an integral.  For $\al\in N\Sigma$, we write $\lfloor\al\rfloor= (\lfloor\al_1\rfloor, \dots,\lfloor\al_m\rfloor)\in N\Sigma\cap\Z^m$. For $\al=(\al^1,\dots,\al^f)\in (\Z^m\cap N\Sigma)^f$, we consider the $mf$-cube of width $\frac 1N$
$$R_{N,\al}:= \{(\la^1,\dots,\la^f)\in (\R^m)^f: \lfloor N\la^j\rfloor = \al^j,\ 1\le j\le f\}\,.$$

Then
\begin{equation} \frac{(m!)^f\,f!}{N^{mf}} \sum_{\sigma \in {\mathcal C}_{N,f}} \log \Pi_{N|\supp\sigma}(z, z) =
\int_{U_N}\log
\sum_{j=1}^f|\wh m^N_{\lfloor N\la^j\rfloor}(z)|^2 \,d\la^1\cdots d\la^f\,,\end{equation}
where  $d\la^j= m!\,d\la^j_1\cdots d\la^j_m$, and $$U_N=\bigcup\left\{R_{N,\al}: \al=(\al^1,\dots,\al^f)\in (\Z^m\cap N\Sigma)^f,\ \al^j\neq \al^{j'} \ \mbox{for\ } j\neq j'\right\}.$$ It then follows from \eqref{CNf} and the estimate $\vol(\Sigma^f\triangle U_N)=O(1/N)$ that \begin{equation}\label{sumtoint} \frac{1}{C(N, f)} \sum_{\sigma \in {\mathcal C}_{N,f}} \log \Pi_{N|\supp\sigma}(z, z) =
\int_{\Sigma^f}\log
\sum_{j=1}^f|\wh m^N_{\lfloor N\la^j\rfloor}(z)|^2 \,d\la^1\cdots d\la^f +E_N(z)\,,\end{equation} where
$$|E_N(z)|\le \frac {C_m}{N}\,\max_{\be^j\in \Z^m\cap N\Sigma}\left|\log
\sum_{j=1}^f|\wh m^N_{\be^j}(z)|^2\right|.$$  As in the proof of Lemma \ref{uniform}, we conclude from Lemma \ref{ptwisebound} that \begin{equation}\label{dblbound}\max_j\{-Nb_{\be^j/N}(z)\}-C'_m \le \log
\sum_{j=1}^f|\wh m^N_{\be^j}(z)|^2 \le\max_j\{-Nb_{\be^j/N}(z)\}+ m\log N +C_m'+\log f.\end{equation}
Therefore, there are positive constants $C,\ C'$ depending only on $m,f$ such that
\begin{equation}\label{ENbound}|E_N(z)| \le C \sup_{\la\in\Sigma} b_\la(z) +C'\,.\end{equation}

\begin{lem}\label{uniformint} Let $\Psi:\Sigma^f\times (\C^*)^m\to \R$ be given by $$\Psi(\la, z)= \log
\sum_{j=1}^f|\wh m^N_{\lfloor N\la^j\rfloor}(z)|^2\,.$$ Then for all compact sets $K\subset  (\C^*)^m$,
$$\frac 1N \Psi(\la, z) \to\max_j\{ -b_{\la^j}(z)\}\qquad \mbox{uniformly on }\ \Sigma^f\times K\,.$$ \end{lem}

\begin{proof} Let $\epsilon>0$ be arbitrary.  By \eqref{dblbound}, we can choose $N_0$ such that $$\left|\frac 1N \log
\sum_{j=1}^f|\wh m^N_{\be^j}(z)|^2 -  \max_j\{ -b_{\be^j/N}(z)\}\right|\le \ep \quad \forall \be\in\Sigma^f,\ \forall z\in (\C^*)^m,\ \forall N\ge N_0\,.$$ We can choose $N_0$ large enough so we also have $$|\al-\la|<\frac 1{N_0} \ \implies\ |b_\al(z)-b_\la(z)|<\epsilon \quad \forall \al,\la\in\Sigma,\ \forall z\in K\,.$$  Thus, for all $(\la,z)\in \Sigma^f\times K$ and $N>N_0$, we have
\begin{multline*}\left|\frac 1N \Psi(\la,z)- \max_j\{-b_{\la^j}(z)\}\right|\le
\left|\frac 1N \Psi(\la,z)- \max_j\{-b_{\lfloor N\la^j\rfloor/N}(z)\right|\\ +\left |\max_j\{-b_{\lfloor N\la^j\rfloor/N}(z)- \max_j\{-b_{\la^j}(z)\}\right| <2\epsilon.\end{multline*}
\end{proof}
 The desired uniform convergence  \eqref{sumlim} follows from \eqref{bx}, \eqref{sumtoint}, \eqref{ENbound}, and Lemma \ref{uniformint}, which completes the proof of Theorem \ref{main}.

The same argument gives the proof of Theoem \ref{mainIV}.

\subsection{Computing the explicit formula: Proof of Corollary \ref{maincor}}\label{explicit}
For $r\in\R^m$, we write $e^r=(e^{r_1},\dots, e^{r_m})$, so that $\sum e^{\rho_j} = \|e^{\rho/2}\|^2 = |e^\rho|$.

Recalling \eqref{bpoint}, we write
$$b(\la;\rho):=b_{\{\la\}}(e^\rho)= \langle
\wh\lambda, \log \wh\lambda \rangle - \langle \rho, \lambda
\rangle +\log\left(1+|e^\rho|\right) \ge 0\,.$$

Therefore,
\begin{equation}\label{reform}
\int_{\Sigma^f}
 \max_{j = 1, \dots, f} \left[\langle \rho, \lambda^j
\rangle -\langle
\wh\lambda^j, \log \wh\lambda^j \rangle  \right]\, d \lambda^1 \cdots d \lambda^f= \log (1+|e^\rho|)-\int_{\Sigma^f}
 \min_{j = 1, \dots, f} b(\la^j;\rho)\, d \lambda^1 \cdots d \lambda^f \,. \end{equation}

We shall use the following elementary probability formula: Let $X$ be a  non-negative random variable  on a probability space $(\Omega,dP)$, and let $D_X(t):=P(X\le t)$ be its distribution function.  The expected value of $X$ is given by
$$\E(X)= \int X\,dP = \int_0^\infty t\,dD_X(t)
=\lim_{r\to\infty}\int_0^r t\,dD_X(t) \,,$$
where
$$\int_0^r t\,dD_X(t) =
rD_X(r) - \int_0^r D_X(t)\,dt = \int_0^r [D_X(r)-D_X(t)]\,dt\,.$$
Letting $r\to\infty$, we have by Lebesgue monotone convergence
\begin{equation}\label{elemprob} \E(X)=\int X\,dP = \int_0^\infty [1-D_X(t)]\,dt\,.\end{equation}

We let $$D_b(t;\rho): = P\{\la\in\Sigma:b(\la;\rho)\le t\}$$ be the distribution function for $b(\cdot;\rho)$, where $dP(\la) = {m!} \,d\la_1\cdots d\la_m$.  The distribution function for the random variable $$X(\la^1,\dots,\la^f):=\min\{b(\la^1),\dots,b(\la^f)\}$$ on $\Sigma^f$ (with the product measure $dP(\la^1)\cdots dP(\la^f)$) is given by $$D_X =1-(1-D_b)^f\,.$$
It then follows from \eqref{reform}--\eqref{elemprob} that
\begin{equation}\label{rereform}  \int_{\Sigma^f}
 \max_{j = 1, \dots, f} \left[\langle \rho, \lambda^j
\rangle -\langle
\wh\lambda^j, \log \wh\lambda^j \rangle  \right]\, d \lambda^1 \cdots d \lambda^f=\log (1+|e^\rho|)-\int_0^\infty \left[1-D_b(t;\rho)\right]^f dt\,. \end{equation}
Corollary \ref{maincor} follows immediately from Theorem \ref{main} and \eqref{rereform}.\qed

\subsubsection{The dimension 1 case} We now further evaluate $D_b$ when the dimension
 $m=1$. In this case, $$b(\la;\rho) = \la\log\la +(1-\la)\log(1-\la) - \rho\la +\log(1+e^\rho)\,,\quad 0\le \la\le 1,\; \rho\in\R\,.$$  Since $b$ is a convex function of $\la$ (taking the minimum value $0$ when $\la= e^\rho/(1+e^\rho)$\,), we have $D_b(t;\rho)= \wt g(t,\rho)-g(t,\rho)$ for $t\ge 0$, where $g(\cdot,\rho)\le \wt g(\cdot,\rho)$ are the branches of $b(\cdot,\rho)\inv$.  Precisely, $g=g(t,\rho),\,\wt g= \wt g(t,\rho)$ are given by   
$$\begin{array}{llll} 0\le g\le \wt g\le 1,\\ b(g;\rho) = t  \quad  \mbox{if } \ t\le \log(1+e^\rho)\,, \quad  & b(g;\rho) =0  & \mbox{if } \ t> \log(1+e^\rho)\,,\\
b(\wt g;\rho) = t  \quad \mbox{if } \ t\le \log(1+e^{-\rho})\,, \quad  & b(\wt g;\rho) =1  & \mbox{if } \ t> \log(1+e^{-\rho})\,.\end{array}$$
We have the symmetry $b(\la;\rho)=b(1-\la;-\rho)$, and hence $\wt g(t,\rho)= 1-g(t,-\rho)$.
Therefore,
\begin{equation}\label{Db} D_b(t;\rho) = 1-g(t,\rho)-g(t,-\rho)\,,\end{equation} where
$g(\cdot,\rho):[0,+\infty)\to [0,e^\rho/(1+e^\rho)]$ is given by:
\begin{equation}\label{lambda1}\begin{array}{ll}
b(g(t,\rho),\rho) = t,\quad &\mbox{if }\ 0\le t
\le \log (1+e^\rho), \\g(t,\rho)=0, &\mbox{if }\  t \ge \log (1+e^\rho).\end{array}\end{equation}

\section{General toric \kahler potentials}\label{TORIC} 

We now sketch the proof of Theorem \ref{maintoric}. It is almost
the same as in the Fubini-Study case but requires the
generalization of Lemma \ref{ptwisebound} and
then Lemma \ref{uniform}.

As discussed in \cite{SoZ}, the toric norming constants can be
written in terms of the symplectic potential as follows:
\begin{equation} \label{INTpf}
Q_{G_N(\phi, dV_{\phi})}(\alpha) = \int_{P}e^{-N (u_{\phi} (x) +
\langle \frac{\alpha}{N} - x, \log \mu_{\phi}^{-1}(x) \rangle} dx.
\end{equation}
Here, $\mu_{\phi}(e^{\rho/2}) = \nabla_{\rho} \phi (e^{\rho/2})$
is the moment map determined by $\phi$.  Applying steepest descent
to the integral, we find that there exists only one critical point
at $x = \mu_{\phi}(e^{\rho/2})$, and we conclude that 
 \begin{equation} \label{LOGNORM} \frac{1}{N} \log Q_{G_N(\phi, dV_{\phi})} (\alpha) = u_{\phi}\left(\frac{\alpha}{N}\right) +O\left(\frac{\log N}N\right)
 \end{equation} uniformly  \cite[(25)]{SoZ3}.

The logarithmic asymptotics (\ref{LOGNORM}) is the only
non-obvious aspect of the logarithmic mass asymptotics.
 The  \szego kernel for a single lattice point (on the diagonal)
 equals
$$\Pi_{N,Q | \alpha}(e^{\rho/2}, e^{\rho/2}) = \frac{e^{\langle \alpha, \rho \rangle} e^{-
N
 \phi(e^{\rho/2})}}{Q_{G_N(\phi, dV_{\phi})}}.$$
The analogue of
Lemma \ref{ptwisebound} for a general \kahler potential is \begin{equation}
\label{bpoint2}  \log \Pi_{N,Q | N x}(e^{\rho/2},
e^{\rho/2})  =N \left(\langle x, \rho \rangle -\phi(e^{\rho/2}) -
u_{\phi}(x )\right) +O(\log N)\,,
\end{equation} which follows from (\ref{LOGNORM}) and \cite[(55)]{SoZ3}.  

For a fewnomial \szego kernel with  a finite set $S$ of lattice
points, the analogue of  Lemma \ref{uniform} is that
\begin{equation} \label{MAX}  \frac{1}{N} \log \Pi_{N,Q | N S}(e^{\rho/2},
e^{\rho/2}) = \max_{\lambda \in S} (\langle \lambda, \rho
\rangle  - u_{\phi}(\lambda )) -\phi(e^{\rho/2}) +O\left(\frac{\log N}{N}\right). \end{equation}
The proof is the same as that of Lemma \ref{uniform}, using (\ref{bpoint2}).  With
this modification, the remainder of the proof of Theorem
\ref{maintoric} is the same as that of Theorem \ref{main}.

\section{The ${\mathbf T}^m$ ensemble}\label{KAC} 

Finally, we indicate the modifications needed to deal with the fewnomial Kac-Hammersley
ensemble (\ref{KHENS}) . This is quite different from the case of pluri-subharmonic
weights because the \szego kernel has quite different (much
weaker) asymptotic properties. But for fewnomial \szego kernels
the distinction is not too severe.

In this case, we use the $L^2$ norm $\| \cdot\|_{{\mathbf T}^m}$ on the
real torus rather than the Fubini-Study norm. We therefore have
\begin{equation*}
\E_{{\mathbf T}^m|S}\left(|P(z)|^2_{{\mathbf T}^m}\right) = \sum_{\alpha,\be
\in S} \E(\la_{\alpha}\bar \la_\be
)\chi_{\alpha}(z)\overline{\chi_{\be}(z)} \;.\end{equation*} Since
 $\E(\la_{\alpha}\bar
\la_\be)=\delta_\al^\be$, we have:
\begin{equation}\label{EszegoS1}
\E_{{{\mathbf T}^m}|S}\left(|P(z)|^2_{{\mathbf T}^m}\right) = \sum_{\alpha \in
S} |\chi_\al(z)|^2 =\Pi_{{{\mathbf T}^m}|S}(z,z)\;,\end{equation} where
 $\Pi_{{{\mathbf T}^m}|S}$ is the orthogonal projection onto $\poly(S) \subset
L^2({(\C^*)^m}, \delta_{{{\mathbf T}^m}})$.  It then follows by
expressing the Gaussian in spherical coordinates that the
expectation in the fewnomial Kac-Hammersley ensemble  is given by
$$\E_{KH}(|P(z)|^2_{{\mathbf T}^m}) =\frac{1}{\#
S}\E_{KH}\left(|P(z)|^2_{{\mathbf T}^m}\right)= \frac{1}{\#
S}\Pi_{{{\mathbf T}^m}|S}(z,z)\;.$$

It is clear that
\begin{equation} \Pi_{\T|S} (z,w) =  \sum_{\alpha \in S} \langle z,\bar w\rangle^{\alpha}\,.
\end{equation}

Therefore,
\begin{equation} \Pi_{\T,S} (z,z) =  \sum_{\alpha \in S}
|z^{\alpha}|^2 = \sum_{\alpha \in S} e^{ \langle \rho, \alpha
\rangle}, \;\; z = e^{i \phi + \rho/2}\,.
\end{equation}

The potential in this case is
\begin{equation} F_N^f (z): =  \frac{1}{C(N, f)}
\sum_{S \in {\mathcal F}_{N,f}} \log \Pi_{\T,S} (z,z)\,.
\end{equation}

\begin{prop}
$$\lim_{N \to \infty} \frac 1N\, F_N^f(e^{\rho/2})  =\int_{\Sigma^f} \max\{\langle
x_1,\rho \rangle  \dots, \langle x_f,\rho \rangle \}\, dx_1 \cdots
dx_f\,. $$
\end{prop}

\medskip\noindent {\it Outline of the proof:\/}
Indeed, \begin{equation} \log \; \sum_{\alpha \in S} e^{ \langle \rho, \alpha
\rangle} = N\log \sum_{\alpha \in S} e^{ \langle \rho,
\al/{N} \rangle}  \sim  N \max_{\alpha \in S}
\{\langle \rho, \al/{N} \rangle\}. \end{equation}
 Hence,
\begin{equation*} F_N^f(e^{\rho})  = \frac{1}{C(N, f)}
\sum_{S \in {\mathcal F}_{N,f}}  \log  \sum_{\alpha \in S} e^{ \langle \rho, \alpha
\rangle}\sim N \int_{\Sigma^f} \max\{\langle
x^1,\rho \rangle  \dots, \langle x^f,\rho \rangle \} dx^1 \cdots
dx^f. \end{equation*}
\qed

\medskip
We note that for each $(x^1, \dots, x^f)$, the function $M_{(x^1,
\dots, x^m)}(\rho) =  \max\{\langle x^1,\rho \rangle \dots,
\langle x^f,\rho \rangle \}$ is a piecewise linear convex function.
It follows  that the integral defines a convex function of $\rho$.

In dimension one, if all $x^j \geq 0$,  $$\max\{\rho x^1, \dots,
\rho x^f) = \left\{
\begin{array}{ll} \rho \max\{x^1, \dots, x^f\}, & \rho \geq 0, \\
& \\ \rho \min \{x^1, \dots, x^f\}, & \rho \leq 0. \end{array}
\right.$$ Hence,
\begin{equation} F_N^k(e^{\rho}) \sim   \left\{
\begin{array}{ll}  N \rho \{\;\int_{[0, 1]^f} \max\{x^1, \dots, x^f\} dx^1 \cdots
dx^f\} , & \rho \geq 0  \\ &  \\  N \rho \{\;\int_{[0, 1]^f}\min
\{x^1, \dots, x^f\} \} dx^1 \cdots dx^f\} , & \rho \leq 0
\end{array} \right.
\end{equation}
Thus, $F_N^k(e^{\rho})$ is piecewise linear in $\rho$ with a
corner at $\rho = 0$.  In dimension one,
\begin{equation}\frac 1N\,\E_{N,f}(Z_{P^N})=\frac {\sqrt{-1}}{2\pi N}\ddbar F_N^k\to
\delta_{S^1}\,.\end{equation}

\end{document}